\documentclass[]{amsart}
\usepackage{geometry}               
\usepackage{amsmath, amssymb, amsthm}
\usepackage{xcolor}

\usepackage[applemac]{inputenc}

\definecolor{myurlcolor}{rgb}{0,0,0.4}
\definecolor{mycitecolor}{rgb}{0,0.5,0}
\definecolor{myrefcolor}{rgb}{0.5,0,0}
\usepackage[pagebackref,colorlinks,linkcolor=myrefcolor,citecolor=mycitecolor,urlcolor=myurlcolor,hypertexnames=true]{hyperref}
\usepackage{amsrefs}
\usepackage{graphicx}
\usepackage{pgfplots}
\usepackage{enumitem}
\usepackage{microtype}
\usepackage{tikz-cd}
\usepgfplotslibrary{fillbetween}

\theoremstyle{plain}
\newtheorem{theorem}{Theorem}[section]

\newtheorem{lemma}[theorem]{Lemma}

\newtheorem{proposition}[theorem]{Proposition}

\theoremstyle{remark}

\newtheorem{remark}[theorem]{Remark}

\theoremstyle{definition}
\newtheorem{definition}[theorem]{Definition}
\newtheorem{example}[theorem]{Example}

\def\C{\mathbb{C}}
\def\R{\mathbb{R}}

\def\N{\mathbb{N}}

\newcommand{\op}[1]{\operatorname{#1}}
\newcommand{\mcal}[1]{\mathcal{#1}}

\newcommand{\her}{\operatorname{Her}}
\newcommand{\frdual}[1]{#1^{\vee_{\op{fr}}}}
\newcommand{\tr}{\op{tr}}

\let\originalleft\left 
\let\originalright\right 
\renewcommand{\left}{\mathopen{}\mathclose\bgroup\originalleft} 
\renewcommand{\right}{\aftergroup\egroup\originalright} 

\title{Abstract Operator Systems over the Cone of Positive Semidefinite Matrices}

\author{Martin Berger}
\address{Martin Berger, Universit\"at Innsbruck, Austria }
\email{martin.berger@uibk.ac.at}

\author{Tim Netzer}
\address{Tim Netzer, Universit\"at Innsbruck, Austria}
\email{tim.netzer@uibk.ac.at}

\pgfplotsset{compat=1.12}
\begin{document}

\begin{abstract} 
There are several important abstract operator systems with the convex  cone of positive semidefinite matrices at the first level. Well-known are the operator systems of separable matrices, of positive semidefinite matrices, and of block positive matrices. In terms of maps, these are the operator systems of entanglement breaking, completely positive, and positive linear maps, respectively. But there exist other interesting and less well-studied such operator systems, for example  those of completely copositive maps, doubly completely positive maps, and decomposable maps, which all play an important role in quantum information theory. We investigate which of these systems is finitely generated, and which admits a finite-dimensional realization in the sense of the Choi-Effros Theorem. We answer this question for all of the described systems completely. Our main contribution is  that decomposable maps form a system which does not admit a finite-dimensional realization, though being finitely generated, whereas the system of doubly completely positive maps is not finitely generated, though having a finite-dimensional realization. This also implies that there cannot exist a finitary Choi-type characterization of doubly completely positive maps.\end{abstract}

\maketitle

\section{Introduction}
Certain types of linear maps between matrix algebras are of great importance in operator algebra and  quantum information theory \cite{pau, wo}. First and foremost, these are {\it completely positive maps},   also known as {\it quantum channels} when they preserve trace. They model a communication channel through which quantum information can be transmitted. But other types of maps are also of interest, for example {\it entanglement breaking} or {\it decomposable maps}. 

Any linear map between two matrix algebras is uniquely determined by its {\it Choi matrix}, a matrix that lives in the tensor product of the two matrix algebras. All different positivity properties of the maps have straightforward translations to notions of positivity of the Choi matrices. For example, completely positive maps correspond to positive semidefinite (psd) matrices, and entanglement breaking maps correspond to separable matrices.

This yields a connection to free (=non-commutative) semialgebraic geometry, were sets of positive matrices are also studied  \cite{hkm, neimn}. The crucial idea  here is to consider sets of matrices {\it of all sizes simultaneously}. This approach often reveals a structure that is invisible when looking at matrices of a fixed size alone, Helton's free version of Hilbert's 17th problem \cite{hepos} being maybe the most striking example. This suggests to look at the above positive maps/matrices also size independently. This can be done in the (almost) equivalent settings of {\it free spectrahedral cones} or {\it abstract operator systems}. 
While free spectrahedral cones play an important role in free semialgebraic geometry,  {\it abstract operator system} do so in operator algebra. It turns out that free spectrahedral cones are precisely the abstract operator systems that admit a finite-dimensional realization in the sense of the Choi-Effros Theorem.

To conclude, quantum information theory, operator algebra and free semialgebraic geometry often study the same objects, just from different perspectives (see also \cite{dlcn}).

In this paper we continue the study of operator systems from the viewpoint of free semialgebraic geometry, and vice versa. We consider abstract operator systems over the convex cone of positive semidefinite matrices, and study some of their properties. In particular, we ask which of these systems admit a finite-dimensional realization, and which are finitely generated. For systems over polyhedral cones, these questions have been answered in \cite{fritz}.  So passing to cones of positive matrices is a natural next step, and also well-motivated  by the above described applications in different areas.

Our contribution in this paper is the following. We start by considering the operator systems of separable matrices, positive semidefinite matrices, matrices with positive semidefinite partial transpose, and block positive matrices. In terms of maps these correspond to  entanglement breaking, completely positive, completely copositive, and positive maps. We combine several results from different areas to obtain the following: Separable and block positive matrices form operator systems that are neither finitely generated nor have a finite-dimensional realization. Positive semidefinite  matrices, and  matrices with positive semidefinite partial transpose, form operator systems that both have a finite-dimensional realization and are finitely generated. {\it Doubly positive matrices} arise as the intersection of positive semidefinite  matrices with matrices of positive semidefinite partial transpose, and they thus form an operator system with a finite-dimensional realization. {\it Decomposable matrice} arise as the sum/convex hull of positive semidefinite  matrices with matrices of positive semidefinite partial transpose, and they thus form an operator system that is finitely generated. Our main result is that the system of doubly positive matrices is not finitely generated, and the system of decomposable matrices does not admit a finite-dimensional realization. The former can also be stated as the impossibility as a finitary Choi-type characterization of double positive matrices/doubly completely positive linear maps.

Our proof uses techniques from semialgebraic geometry. Namely, we show that already the finite level sets of the operator system of decomposable matrices cannot be defined by  linear matrix inequalities. Thus the whole operator system of decomposable maps cannot be a free spectrahedron, i.e.\ cannot have a finite-dimensional realization. Duality then yields the result for doubly positive matrices. 
We do so by explicitly computing the intersection of decomposable matrices with a two-dimensional subspace, and show that the arising set fails to have two necessary properties for having a linear matrix inequality representation.
Note that although this sounds straightforward, the problem lies in the fact that decomposability of a matrix is not easy to determine. Even if the matrix is of a special form, i.e.\ lies in a certain subspace,  the two matrices that might provide a decomposition will in general not be of the same form again. Figure \ref{fig:pluspsd} below indeed shows that decomposable matrices in our subspace are much more general than convex combinations of positive semidefinite matrices and matrices with positive semidefinite partical transpose within our subspace.

\section{Preliminaries}
Let us first fix notation and introduce the basic concepts and setup we need to prove our results.
\subsection{Operator Systems and Free Spectrahedra}
Throughout, let $d,s,t\in\N$. We write ${\rm Mat}_{d,s}(\C)$ for the space of complex $d\times s$ matrices and set ${\rm Mat}_d(\C):={\rm Mat}_{d,d}(\C),$ where the identity matrix is denoted by $I_d$. Furthermore let $\mathcal{V}$ denote a $\C$-vector space with involution $*$, and $\mathcal{V}_h$ the $\R$-subspace of its Hermitian elements. For any $s\geq 1$ the vector space ${\rm Mat}_s(\mathcal{V}) = \mathcal{V}\otimes_{\C}{\rm Mat}_s(\C)$ of $s\times s$-matrices with entries from $\mcal{V}$ comes equipped with the canonical involution defined by $\left(v_{ij}\right)_{i,j}^*:=\left(v_{ji}^*\right)_{i,j}$.

\begin{definition}[{e.g.~\cite[Chapter~13]{pau}}]
An {\it abstract operator system} $\mathcal{C}$ on $\mathcal{V}$ consists of a closed and salient convex cone $\mathcal{C}_s\subseteq {\rm Mat}_s(\mathcal{V})_h$ for each $s\geq 1$, and some $u\in \mathcal{C}_1\subseteq \mathcal{V}_h$ such that
\begin{itemize}
\item[($i$)] $A\in \mathcal{C}_s, V\in {\rm Mat}_{s,t}(\mathbb C) \Rightarrow V^*AV\in \mathcal{C}_t$,
\item[($ii$)]  $u\otimes I_s$ is an order unit of $\mathcal{C}_s$ for all $s\geq 1$.
\end{itemize}
\end{definition}

\begin{remark}\label{rem: opsystem}
\hangindent\leftmargini
\textup{($i$)}\hskip\labelsep The topology in which each $\mathcal{C}_s$ is required to be closed is understood to be the finest locally convex topology on $\mathcal V$.
\begin{enumerate}
\setcounter{enumi}{1}
\item[($ii$)] Being an order unit is equivalent to being an interior point, see \cite{cimp}.
\item[($iii$)] $u\otimes I_s$ is an order unit  of $\mathcal{C}_s$ for all $s$ if and only if this holds for $s=1$, see \cite{fritz}.
\item[($iv$)] The condition of $\mcal{C}_s$ being closed is equivalent to $u$ being an \textit{archimedean} order unit, again \cite{cimp,pau}.
\item[($v$)] We refer to $\mcal{C}_s$ as the $s$-th level of $\mcal{C}$ and write $A\in\mcal{C}$ if there exists an $s\geq 1$ such that $A\in \mcal{C}_s$.
\end{enumerate}
\end{remark}
By the Choi--Effros Theorem (\cite{choi}, see  also \cite[Chapter~13]{pau}), for every abstract operator system $\mathcal{C}$  there exists a Hilbert space $\mathcal H$ and a $*$-linear mapping $\varphi\colon \mathcal V\rightarrow \mathbb B(\mathcal H)$ with $\varphi(u)={\rm id}_{\mathcal H}$, such that for all $s\geq 1$ and $A\in \mathcal{C}_s$,
$$
A\in \mathcal{C}_s \:\Leftrightarrow\: (\varphi\otimes{\rm id})(A)\geqslant 0,
$$
where $\geqslant 0$ denotes positive semidefiniteness (psd).
On the right-hand side, we use the canonical identification
$$
{\rm Mat}_s(\mathbb B(\mathcal H))=\mathbb B(\mathcal H)\otimes_\C {\rm Mat}_s(\C) = \mathbb B(\mathcal H^s)
$$
to define positivity of the operator. Such a mapping $\varphi$ is called a {\it concrete realization} or just {\it realization} of the operator system $\mathcal{C}$. A realization $\varphi$ is necessarily injective, since $\mathcal{C}_1$ does not contain a nontrivial subspace.

\begin{definition}
An abstract operator system $\mathcal{C}$ is {\it finite-dimensional realizable} if there is a realization with $\dim{\mathcal{H}}<\infty$. \end{definition}
From now on we will {\it always assume that $\mathcal V$ is finite-dimensional}. After a suitable choice of basis we can then even assume $\mathcal V=\mathbb C^d$ with the canonical involution, and thus  $\mathcal V_h=\R^d$. Then
$$
{\rm Mat}_s(\mathcal V)=\mathcal V\otimes_\C {\rm Mat}_s(\C)={\rm Mat}_s(\C)^d,\quad {\rm Mat}_s(\mathcal V)_h=\her_s(\C)^d,
$$
where $\her_s(\C)$ denotes the Hermitian $s\times s$ matrices over $\C,$ and a realization of $\mathcal{C}$ just consists of self-adjoint operators $T_1,\ldots, T_d\in\mathbb B(\mathcal H)_h$ with $$u_1T_1+\cdots+u_dT_d={\rm id}_{\mathcal H}$$ and 
$$
(A_1,\ldots, A_d)\in \mathcal{C}_s \:\Leftrightarrow\: T_1\otimes A_1+\cdots + T_d\otimes A_d\geqslant 0.
$$
 Finite-dimensional realizability then means thus the $T_i$ can be taken to be Hermitian matrices.

\begin{definition}
A {\it (classical) spectrahedron} \cite{Ramana1995} is a set of the form
$$
\left\{ a\in\R^d \bigm|  a_1B_1+\cdots+a_dB_d\geqslant 0\right\},
$$
where $B_1,\ldots, B_d\in \her_r(\C)$ are Hermitian matrices.  For any $s\geq 1$, we define
$$
\mathcal S_s(B_1,\ldots, B_d) := \left\{ (A_1,\ldots, A_d)\in\her_s(\C)^d \bigm| B_1\otimes A_1+\cdots +B_d\otimes A_d\geqslant 0\right\}.
$$
The family of cones $\mathcal S(B_1,\ldots,B_d)=\left(\mathcal S_s(B_1,\ldots, B_d)\right)_{s\geq 1}$ is called the {\it free spectrahedron} defined by $B_1,\ldots,B_d$.
\end{definition}
\begin{remark}
In order for a free spectrahedron to be an operator system, the positive cones must be salient and have an order unit. The first is equivalent to the $B_i$ being linearly independent, and the latter happens in particular if there is $u\in\R^d$ with $\sum_i u_i B_i =  I_r$, in which case we take such  $u$ to be the order unit.
\end{remark}
Spectrahedra are the feasible sets of semidefinite programming, see \cite{blekherman2013} for an overview. Recall that a semialgebraic set $S\subseteq\R^d$ is called \textit{basic closed} if there exist polynomials $p_1,\ldots,p_r$ such that $S=\{a\in\R^d\,|\, p_1(a)\geq 0,\ldots, p_r(a)\geq 0\}$. One sees easily that for every free spectrahedron each level is a classical spectrahedron, and that every classical spectrahedron is a closed, convex and basic closed semialgebraic set. For example, the principal minors of the linear matrix polynomial $x_1B_1+\ldots+x_n B_n$ define $\mathcal S_1(B_1,\ldots, B_d)$ as a basic closed semialgebraic set in the above situation. 

\subsection{Largest and Smallest Operator Systems, Free Duality}
\label{subsec: cones}
In the following let $C\subseteq\R^d$ denote a closed salient cone with order unit $u,$ and $\mcal{C}=(\mcal{C}_s)_{s\in\N}$ an operator system such that $\mcal{C}_s\subseteq\her_s(\mathbb C)^d$ for every $s\geq 1$. 
We are interested in operator systems $\mcal{C}$ such that $\mcal C_1=C$. It is known that the largest such operator system $C^{\max}$ is given via
$$
C_s^{\max}:=\left\{ (A_1,\ldots, A_d)\in{\rm Her}_s(\mathbb C)^d \bigm| \forall v\in \mathbb C^s: \ (v^*A_1v,\ldots, v^*A_dv)\in C\right\}.
$$
while the smallest such operator system $C^{\min}$ has the form
$$
C_s^{\min}:=\left\{\sum_{i=1}^n c_i\otimes P_i \biggm| n\in\N,c_i\in C, P_i\in{\rm Her}_s(\mathbb C), P_i\geqslant 0\right\},
$$ 
see \cite{pauto,fritz}. The above system is largest in the sense that for any operator system $\left(  \mcal D_s\right)_{s\geq 1}$ with $ \mcal D_1\subseteq C$, we have $ \mcal D_s\subseteq  C_s^{\max}$ for all $s\geq 1,$ and analogously for the smallest. Recall that an inner product on $\her_s(\C)^d$ is given as 
$$
\langle A,B\rangle=\sum_{i=1}^d\tr(B_i A_i)
$$
for $A=(A_1,\ldots,A_d),B=(B_1,\ldots,B_d)\in\her_s(\C)^d,$ where $\tr$ denotes the trace of a matrix. Thus the (classical) dual of a cone $\mcal{C}_s\subseteq\her_s(\C)^d$ is given via
$$
\mcal C_s^\vee=\left\{(A_1,\ldots,A_d)\in\her_s(\C)^d \biggm| \sum_{i=1}^d\tr(B_i A_i)\geq 0\text{ for every } (B_1,\ldots,B_d)\in\mcal C_s\right\}.
$$
In contrast, we define the \textit{free dual} $\frdual{\mcal{C}}$ of an operator system $\mcal{C}$ by
$$
\frdual{\mcal{C}_s}:=\left\{(A_1,\ldots,A_d)\in\her_s(\C)^d\,\biggm|\,\sum_{i=1}^d  B_i^T\otimes A_i\geqslant 0\text{ for every } (B_1,\ldots,B_d)\in\mcal{C}\right\}.
$$
Note that in this last definition $(B_1,\ldots, B_d)$ runs through all of $\mcal C$, i.e.\ through all levels $\mcal C_t$. Also, $B_i^T$ denotes the matrix $B_i$ transposed.

\begin{proposition}\label{prop: freedual}
Let $C\subseteq\R^d$ denote a closed salient cone with order unit, and let $\mcal{C}$ and $\mcal{D}$ be operator systems with $\mcal C_s,\mcal D_s\subseteq\her_s(\C)^d$ for every $s\geq 1$. Then the following holds:
\begin{enumerate}[label=(\roman*)]
\item $\frdual{\mcal{C}}_s\subseteq \mcal{C}_s^\vee$ for every $s$, where equality
 holds if $\mcal{C}=C^{\min}$ or $s=1$.\label{prop:freedual_dual}
 \item $\frdual{\mcal{C}}$ is an operator system.
 \item $\frdual{(\frdual{\mcal{C}})}=\mcal{C}$. \label{prop:freedual_dualdual}
\item $\frdual{(C^{\min})}=(C^\vee)^{\max}$ and $\frdual{(C^{\max})}=(C^\vee)^{\min}$.\label{prop:freedual_minmax}
\item Let $\mcal{C}+\mcal{D}$ and $\mcal{C}\cap\mcal{D}$ denote the (level-wise) Minkowski sum and intersection respectively. Then $$\frdual{(\mcal{C}+\mcal{D})}=\frdual{\mcal{C}}\cap\frdual{\mcal{D}}.$$\label{prop:freedual_sum}
\end{enumerate}
\end{proposition}
\begin{proof}
($i$)  Let $A=(A_1,\ldots,A_d)\in\frdual{\mcal{C}_s}$ and $(B_1,\ldots,B_d)\in\mcal{C}_s$. Set $v=\sum_{j=1}^s e_j\otimes e_j,$ where $e_j\in\C^s$ denotes the $j$-th standard basis vector. Then 
$$
0\leq v^*\left(\sum_{i=1}^d  B_i^T\otimes A_i\right) v=\sum_{i=1}^d\tr(B_i A_i)
$$
holds, and thus $A\in\mcal{C}_s^\vee$. 

Now let $\mcal{C}=C^{\min}$ and take $B=\sum_{j=1}^n c_j\otimes P_j\in C^{\min}_s$ arbitrary, where upon writing $c_j=(c_{j1},\ldots,c_{jd})$ we obtain 
$$
B=\sum_{j=1}^n\left(c_{j1} P_j,\ldots,c_{jd} P_j\right).
$$
Thus $A=(A_1,\ldots,A_d)\in\frdual{(C^{\min})}$ if and only if
$$
0\leqslant \sum_{i=1}^d\left(\sum_{j=1}^n c_{ji}  P_j^T\right) \otimes A_i=\sum_{j=1}^n P_j^T \otimes \left(\sum_{i=1}^dc_{ji} A_i\right)
$$
holds for all $n\in\N, c_1,\ldots,c_n\in C$ and psd matrices $P_1,\ldots, P_n$, which is equivalent to 
\begin{equation}\label{eq: minfrdual}
\sum_{i=1}^d c_i A_i\geqslant 0\qquad\text{for every}\qquad c=(c_1,\ldots,c_d)\in C.
\end{equation}
Now for  $(A_1,\ldots,A_d)\in(C^{\min}_s)^\vee,$ every psd matrix $P\in\her_s(\C),$ and $c=(c_1,\ldots,c_d)\in C$ it holds
$$
0\leq \sum_{i=1}^d\tr\left((c_i P) A_i\right)=\tr\left(P \left(\sum_{i=1}^d  c_i  A_i\right)\right).
$$
By self-duality of the psd cone, this is also equivalent to (\ref{eq: minfrdual}), and therefore $A\in\frdual{(C^{\min})}$.

Finally, in case  $s=1$, take $a\in \mcal C_1^\vee$. Then for every $t\geq 1, B\in \mcal C_t$ and $v\in\C^t$ we obtain  $v^* B^T v\in \mcal  = v^TB\overline v\in C_1,$ and thus
$$
0\leq \sum_{i=1}^d v^*  B_i^T v a_i=v^*\left(\sum_{i=1}^d  B_i^T \otimes a_i\right)v,
$$
so clearly $a\in \frdual{C}_1$.

($ii$) Clearly each $\frdual{\mcal{C}}_s$ is a convex cone. It is  also closed, since 
the cone of psd matrices is closed. Now let $A=(A_1,\ldots,A_d)\in\frdual{\mcal{C}}_s\cap(-\frdual{\mcal{C}}_s),$ which means that  $\sum_{i=1}^d B_i^T\otimes A_i=0$  for every $B\in \mcal{C}.$ Now choose an order unit/interior point $u\in \mcal{C}_1$. Then we find $\varepsilon >0$ such that $u+\varepsilon e_j\in \mcal{C}_1$ for every standard basis vector $e_j\in\C^d$. We obtain 
$$
0= \sum_{i=1}^d (u+\varepsilon e_j)_i\otimes A_i= \varepsilon A_j,
$$
which implies $A_j=0$ for every $j=1,\ldots,d$. Thus $\frdual{C}_s$ is a salient cone. Given $A\in\frdual{\mcal{C}}_s$ and $V\in{\rm Mat}_{s,t}(\C),$ it follows immediately that $V^*A V\in\frdual{\mcal{C}}_t$. It remains to show that $\frdual{\mcal{C}}$ has an order unit. By Remark \ref{rem: opsystem} it is enough to prove existence of an interior point of $\frdual{\mcal{C}}_1$. But since the (classical) dual of a closed salient cone in $\R^d$ has nonempty interior, so does $\frdual{\mcal{C}}_1$ by ($i$).

($iii$) It is easy to see that $\mcal{C}\subseteq \frdual{(\frdual{\mcal{C}})}$. Now let $A=(A_1,\ldots,A_d)\in\her_s(\C)^d\setminus\mcal{C}_s$. By the Effros-Winkler separation theorem \cite{effros}, we find $B_1,\ldots,B_d\in\her_s(\C)$ such that
$$
\sum_{i=1}^dA_i\otimes  B_i^T\ngeqslant 0\qquad\text{and}\qquad \mcal{C}\subseteq\mcal{S}(B_1^T,\ldots, B_d^T).
$$
Then $(B_1,\ldots,B_d)\in\frdual{\mcal{C}}$ but   $A\notin\frdual{(\frdual{\mcal{C}})}$. Note that a non-conic analog of this result can also be found in \cite{helton}.

($iv$) We have 
\begin{align*}
A\in (C^\vee)_s^{\max} & \Longleftrightarrow \forall v\in\C^s:(v^*A_1v,\ldots,v^*A_dv)\in C^\vee\\
& \Longleftrightarrow \forall v\in\C^s\,\forall c\in C: \sum_{i=1}^d c_i v^* A_i v\geq 0\\
& \Longleftrightarrow \forall c\in C: \sum_{i=1}^d c_i A_i\geqslant 0\\
& \Longleftrightarrow  A\in \frdual{(C^{\min})}_{s}
\end{align*}
where the last equivalence is again due to (\ref{eq: minfrdual}). The second statement follows immediately by applying the first one and ($iii$) to $C^\vee.$
Statement ($v$)  is obvious.
\end{proof}

\begin{definition}
An operator system $\mathcal C$ is {\it finitely generated}, if there are finitely many elements $A^{(1)},\ldots, A^{(n)}\in\mathcal C,$ such that $\mathcal C$ is the smallest operator system containing all the $A^{(j)}$. 
\end{definition}

\begin{remark}
($i$) If an operator system is finitely generated, then it is already generated by one element $A\in\mathcal C$. Indeed, a block-diagonal sum $A^{(1)}\oplus \cdots \oplus A^{(n)}$ belongs to $\mathcal C$ if and only if each $A^{(j)}$ belongs to $\mathcal C$. This follows easily from the operator system axioms.

($ii$) The operator system $\mathcal C$ is generated by $A\in\mathcal C_s$ if and only if each element from $\mathcal C$ is of the form $$\sum_i V_i^* A V_i=V^*(A\oplus \cdots \oplus A)V$$ for complex matrices $V, V_i$ (of the correct size). We leave this as an exercise for the reader.
\end{remark}

The following result  was proven in  \cite{helton}, in a non-conic version. For completeness we sketch the proof adapted to our setup.
\begin{theorem}\label{thm: helton}
The operator system $\mathcal C$ is generated by $A=(A_1,\ldots A_d)$ if and only if $\mathcal C^{\vee_{\rm fr}}=\mathcal S(A_1^T,\ldots, A_d^T).$ In particular, $\mcal C$ is finitely generated  if and only if  $\mathcal C^{\vee_{\rm fr}}$ is finite-dimensional realizable, and vice versa.
\end{theorem}
\begin{proof}
If $\mathcal C$ is generated by $A$, then $B\in \mathcal C^{\vee_{\rm fr}}$ if and only if $\sum_i A_i^T\otimes B_i\geqslant 0$, i.e.\ $B\in \mathcal S(A_1^T,\ldots, A_d^T)$. Applying the same argument to the operator system generated by $A$, and then  using  Proposition \ref{prop: freedual} ($iii$)  proves the other direction.
\end{proof}

\begin{example}\label{ex:fritz} We review the results of \cite{fritz} in the context of  Theorem \ref{thm: helton} and Proposition \ref{prop: freedual}. 
So let $C\subseteq\mathbb R^d$ be a salient closed convex cone. We then have  $$C \mbox{ polyhedral} \Leftrightarrow C^{\max}  \mbox{ finite-dimensional realizable} \Leftrightarrow C^{\min} \mbox{ finitely generated}.$$ 
If $C$ is already assumed to be polyhedral, then \begin{align*}C \mbox{ simplex cone} &\Leftrightarrow C^{\min} \mbox{ finite-dimensional realizable } \Leftrightarrow  C^{\max}\ \text{ finitely generated } \\ &\Leftrightarrow C^{\min}=C^{\max}.\end{align*}
In the following we will examine operator systems with important non-polyhedral cones at level one.
\end{example}

\subsection{Operator Systems Over the Cone of Positive Matrices}
\label{subsec: psdcones}
In the following let $\mcal{V}={\rm Mat}_d(\C)$ with the usual involution $*$, thus $\mcal{V}_h=\her_d(\C)$, and let $C\subseteq\her_d(\C)$ be the cone of positive semidefinite complex $d\times d$ matrices. This cone gives rise to several operator systems  with $C$ at level one. 

The smallest such operator system $C^{\min}$ in this situation is also denoted by $\op{Sep}_d=(\op{Sep}_{d,s})_{s\in\N}$ where
$$
\op{Sep}_{d,s}:=\left\{\sum_{i=1}^n A_i\otimes B_i\in{\rm Mat}_d(\C)\otimes{\rm Mat}_s(\C)\biggm| n\in\N,A_i\geqslant 0, B_i\geqslant 0\right\},
$$
and is called the operator system of {\it separable matrices}, see \cite{hor}. The maximal system $C^{\max}$ is known as the operator system of {\it block positive matrices} $\op{Bpsd}_{d}=(\op{Bpsd}_{d,s})_{s\in\N}$, where each level is defined as
\begin{align*}
\op{Bpsd}_{d,s}:=\Bigg\{\sum_{i=1}^n A_i\otimes B_i \biggm|n\in\N,(x\otimes y)^*\bigg(\sum_{i=1}^n A_i\otimes B_i\bigg)(x\otimes y)\geq 0\  \forall  x\in\C^d, y\in\C^s \Bigg\}.
\end{align*}
 In between the separable and block positive systems lie the operator system of positive semidefinite matrices $\op{Psd}_d=(\op{Psd}_{d,s})_{s\in\N}$ and the operator system of matrices with positive partial transpose  $\op{Psd}_d^\Gamma=(\op{Psd}^{\Gamma}_{d,s})_{s\in\N}$, where 
$$
\op{Psd}_{d,s}:=\left\{\sum_{i=1}^n A_i\otimes B_i\in{\rm Mat}_d(\C)\otimes{\rm Mat}_s(\C)\biggm|n\in\N, \sum_{i=1}^n A_i\otimes B_i\geqslant 0\right\}
$$
and
$$\op{Psd}^{\Gamma}_{d,s}:=\left\{\sum_{i=1}^n A_i\otimes B_i\in{\rm Mat}_d(\C)\otimes{\rm Mat}_s(\C)\biggm| n\in\N,\sum_{i=1}^n A_i^T\otimes B_i\geqslant 0\right\}
$$
respectively. The mapping 
$$
\Gamma\colon {\rm Mat}_d(\C)\otimes{\rm Mat}_s(\C)\rightarrow {\rm Mat}_d(\C)\otimes{\rm Mat}_s(\C)\colon \sum_{i=1}^n A_i\otimes B_i\mapsto \sum_{i=1}^n A_i^T\otimes B_i
$$
is called the {\it partial transpose} and we write $X^\Gamma$ for $\Gamma(X)$ where $X\in  {\rm Mat}_d(\C)\otimes{\rm Mat}_s(\C)$. 

The level-wise intersection of two operator systems with the same cone at level one is clearly again such an operator system. 
Thus we obtain the system $\op{Dpsd}_{d}=(\op{Dpsd}_{d,s})_{s\in\N}$ where $$\op{Dpsd}_{d,s}:=\op{Psd}_{d,s}\cap\op{Psd}^\Gamma_{d,s}.$$ We call this the operator system of \textit{doubly positive matrices}. 

Now the operator system of \textit{decomposable matrices} $\op{Decomp}_{d}=(\op{Decomp}_{d,s})_{s\in\N}$ is defined as the level-wise Minkowski sum of the systems $\op{Psd}_d$ and $\op{Psd}_d^\Gamma$, i.e.\ by
$$
\op{Decomp}_{d,s}:=\left\{ X+Y\big|X\in\op{Psd}_{d,s},Y\in\op{Psd}^{\Gamma}_{d,s}\right\}\,.
$$
Using Lemma \ref{lem: sumclosed} from the appendix it follows immediately that $\op{Decomp}_d$ is indeed an operator sytem. It is well known that for $(d,s)\in\lbrace (2,2),(2,3),(3,2)\rbrace$ one has $\op{Decomp}_{d,s}=\op{Bpsd}_{d,s}$, see \cite{stormer,wor}. Otherwise decomposable matrices form a strict subset of the block positive matrices, see again \cite{wor} and \cite{tang} for an explicit counterexample. 
  See Figure \ref{fig: cones} for a first schematic overview over the described operator systems.

There is a connection between the above mentioned cones and cones of linear maps between matrix spaces. By the \textit{Choi-Jamio\l kowski} isomorphism \cite{jami,choimatrix}, every linear map $T\colon {\rm Mat}_d(\C)\rightarrow {\rm Mat}_s(\C)$ is uniquely defined by its \textit{Choi matrix} 
\begin{equation}\label{eq: choimatrix}
C_T:=\sum_{i,j=1}^d E_{ij}\otimes T(E_{ij})\in {\rm Mat}_d(\C)\otimes{\rm Mat}_s(\C),
\end{equation}
where $E_{ij}\in{\rm Mat}_d(\C)$ denotes the $ij$-th matrix unit, i.e.\ the matrix whose entry at position $(i,j)$ is one while all the others are zero. 

A map with a separable Choi matrix is called \textit{entanglement-breaking}, see \cite{hor}. 

The map $T$ is called {\it positive}  if it maps psd matrices to psd matrices. It is easy to check that $T$ is positive if and only if  its Choi matrix is  block positive, i.e.\ $C_T\in\op{Bpsd}_{d,s}.$

By a well known-theorem of Choi \cite{choimatrix}, $T$ is \textit{completely positive}, i.e.\ $ {\rm id}_n\otimes T$ is positive for every $n$, if and only if the  Choi matrix  is psd, i.e.\ $C_T\in \op{Psd}_{d,s}$.

Let 
$$
\vartheta_d\colon{\rm Mat}_d(\C)\rightarrow{\rm Mat}_d(\C)\colon M\mapsto M^T
$$
denote matrix transposition. Then $T$ is called \textit{completely copositive}, if $T\circ\vartheta_d$ is completely positive, which is equivalent to $C_T\in\op{Psd}^{\Gamma}_{d,s}$.

Not surprisingly, $T$ is called {\it doubly completely positive} if $T$ is both completely positive and completely copositive, which just means $C_T\in  \op{Dpsd}_{d,s}.$

Finally, $T$ is called {\it decomposable} if it is the sum of a completely positive and a completely copositive map, which is by the above identification equivalent to $C_T\in\op{Decomp}_{d,s}$. 

Therefore each of the above defined cones/operator systems corresponds to a system of positive linear maps. 

\begin{figure}
\begin{equation*}
\begin{tikzcd}[column sep=0.1em]
\vdots&\vdots&\vdots&&\vdots&&\vdots&&\vdots \\
\op{Sep}_{d,3} \quad\subseteq \arrow[u,dash] & \quad\op{Dpsd_{d,3}}\quad\subseteq \arrow[u,dash]  & \op{Psd}_{d,3} \arrow[u,dash] & +  &\op{Psd}^\Gamma_{d,3} \arrow[u,dash]& = &\quad \op{Decomp}_{d,3} \arrow[u,dash] \quad \subseteq& {}& \op{Bpsd}_{d,3}\arrow[u,dash]\\
\op{Sep}_{d,2}\quad \arrow[u,dash]  \subseteq &\quad\op{Dpsd_{d,2}}\quad\subseteq  \arrow[u,dash]  &  \op{Psd}_{d,2} \arrow[u,dash]& + & \op{Psd}^\Gamma_{d,2} \arrow[u,dash]&= & \quad\op{Decomp}_{d,2} \arrow[u,dash] \quad\subseteq &  {} & \op{Bpsd}_{d,2}\arrow[u,dash]\\[20pt]
&&&&C=\op{Psd}_{d,1} \arrow[ullll,dash] \arrow[ulll,dash]\arrow[ull,dash] \arrow[u,dash] \arrow[urr,dash]\arrow[urrrr,dash]&&&&
\end{tikzcd}
\end{equation*}
\caption{Schematic representation of the relationships between the cones discussed in Section \ref{subsec: psdcones}.}
\label{fig: cones}
\end{figure}
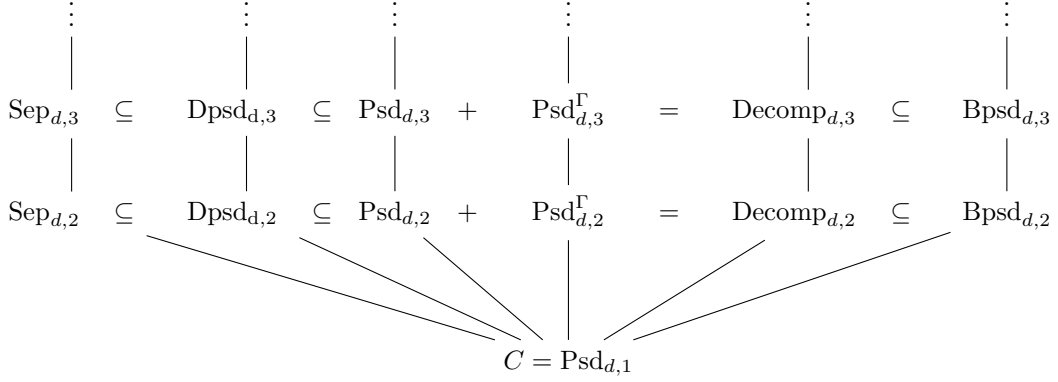

In the remainder of this section we will collect several easy or  known results on the operator systems that we have now introduced. First note that for $d=1$ all the above systems coincide, at level $s$ we just obtain the cone of  psd matrices of size $s$. So we will restrict to $d\geq 2$ from now on.

\begin{remark}\label{rem:warning}
A word of warning is appropriate before we proceed. For the definition of duals and free duals in Section \ref{subsec: cones} we have expressed elements as $d$-tuples, by choosing an orthonormal basis of $\mcal{V}_h$. The level-wise dual $\mathcal C_s^{\vee}$ as well as the free dual $\mathcal C^{\vee_{\rm fr}}$ are defined via this choice. On $C=\op{Psd}_{d,1}$ we will always use the trace inner product. But if $A\in {\rm Her}_d(\C)\otimes {\rm Her}_s(\C)$ is written as $$A=(A_{ij})_{i,j=1,\ldots, d}=\sum_{i,j=1}^d E_{ij}\otimes A_{ij}\in {\rm Her}_{ds}(\C)$$ then this is not a valid basis expansion, since the $E_{ij}$ do not form an orthonormal basis of ${\rm Her}_d(\C)$. However, by expressing everything w.r.t.\ to an orthonormal basis instead,  one immediately checks that  the inner product used for the definition of the level-wise dual in Section  \ref{subsec: cones} does coincide with the usual trace inner product on ${\rm Her}_{d}(\C)\otimes{\rm Her}_{s}(\C)={\rm Her}_{ds}(\C)$. The "tensor inner product" used for the definition of the free dual becomes $$\sum_{ij} B_{ij}^T\otimes A_{ji}$$ for $A=(A_{ij})_{i,j=1,\ldots, d}\in {\rm Her}_d(\C)\otimes{\rm Her}_s(\C)$ and  $B=(B_{ij})_{i,j=1,\ldots, d}\in {\rm Her}_d(\C)\otimes{\rm Her}_t(\C).$
\end{remark}

\begin{proposition}\label{prop:dualsys}
($i$) We have  $\frdual{\op{Sep}_d}=\op{Bpsd}_d$ and $\frdual{\op{Bpsd}_d}=\op{Sep}_d$. Both systems are neither finitely generated nor finite-dimensional realizable.

($ii$) $\op{Psd}_d^{\vee_{\rm fr}}=\op{Psd}_d,$ $(\op{Psd}^{\Gamma}_d)^{\vee_{\rm fr}}=\op{Psd}^{\Gamma}_d,$ and both systems are  both finitely  generated and finite-dimensional realizable. 

($iii$) $\op{Dpsd}_d^{\vee_{\rm fr}}=\op{Decomp}_d$ and  $\op{Decomp}_d^{\vee_{\rm fr}}=\op{Dpsd}_d.$ Furthermore the system $\op{Dpsd}_d$ is finite-dimensional realizable, the system  $\op{Decomp}_d$ is finitely generated.
\end{proposition}
\begin{proof}
($i$) The duality follows from  Proposition \ref{prop: freedual} ($iv$), together with self-duality of $C=\op{Psd}_{d,1}.$ Since $C$ is not polyhedral, the maximal system $\op{Bpsd}_d$ is not finite-dimensional realizable \cite{fritz}, c.f.\ Example \ref{ex:fritz}. By  Theorem \ref{thm: helton}  $\op{Sep}_d$ is thus not finitely generated.

For $d+s>5$ it has been shown  \cite{fawzi} that  $\op{Sep}_{d,s}$ is not a (classical) spectrahedral shadow, i.e.\ not the linear image of a (classical) spectrahedron. 
In particular it is not a spectrahedron, and in particular $\op{Sep}_d$ is not a free spectrahedron, i.e.\ not finite-dimensional realizable. By  Theorem \ref{thm: helton}  $\op{Bpsd}_d$ is not finitely generated.

($ii$) From Proposition \ref{prop: freedual} ($i$) we get  $\left(\op{Psd}_d^{\vee_{\rm fr}}\right)_s\subseteq \op{Psd}_{d,s}^\vee= \op{Psd}_{d,s}$ for each $s\geq 1$, so $\op{Psd}_d^{\vee_{\rm fr}}\subseteq \op{Psd}_d.$ For the other inclusion let $A=\sum_{i,j=1}^dE_{ij}\otimes A_{ij}\in\op{Psd}_{d,s}$ and $B=\sum_{i,j=1}^dE_{ij}\otimes B_{ij}\in\op{Psd}_{d,t}$. Then $$B^T\otimes A=\sum_{i,j,k,l}(E_{ij}^T\otimes E_{kl})\otimes (B_{ij}^T\otimes A_{kl})\geqslant 0$$ is psd, and upon  compressing the two tensor factors on the left with $\sum_{r=1}^de_r\otimes e_r$ we obtain the psd matrix  $$\sum_{i,j} B_{ij}^T\otimes A_{ji}\geqslant 0.$$ But this exactly means that $A\in \op{Psd}_{d}^{\vee_{\rm fr}}$ c.f.\ Remark \ref{rem:warning}. 

 Since  $\op{Psd}_{d}^{\Gamma}$ arises from   $\op{Psd}_{d}$ by the level-wise partial transpose map $\Gamma$, the free self-duality follows easily from the statement for $\op{Psd}_d$.
 
By the very definition, $\op{Psd}_d$ is finite-dimensional realizable (c.f.\ Remark \ref{rem:choi}). By self-duality and Theorem \ref{thm: helton} it is thus also finitely generated. The same is true for  $\op{Psd}_d^{\Gamma}.$

($iii$) The duality statement is immediate from ($ii$) and Proposition  \ref{prop: freedual} ($v$).  As the intersection of two systems with a finite-dimensional realization, $\op{Dpsd}_d$ has a finite-dimensional realization as well (e.g.\ the block diagonal sum realization), and as the Minkowski sum of two finitely generated systems, $\op{Decomp}_d$ is again finitely generated.
 \end{proof}

\begin{remark}\label{rem:choi}
Theorem   \ref{thm: helton}, when applied to $\op{Psd}_d$ as in the last proof, recovers Choi's characterization of all completely positive maps as of the form $A\mapsto \sum_i V_i^*AV_i$. Indeed when choosing the orthonormal basis $$E_{ii}, \frac{1}{\sqrt{2}}\left(E_{ij}+E_{ji}\right), \frac{i}{\sqrt{2}}\left(E_{ij}-E_{ji}\right)$$ of ${\rm Her}_d(\C)$, these are precisely the coefficient matrices defining $\op{Psd}_d$ as a free spectrahedron (and thus give rise to a finite-dimensional realization). So the dual, $\op{Psd}_d$ itself, is generated by the corresponding tuple of transposed matrices, which is easily checked to be $$\sum_{i,j}E_{ij}\otimes E_{ij},$$  the Choi matrix of the map ${\rm id}_d\colon {\rm Mat}_d(\C)\to{\rm Mat}_d(\C)$.
\end{remark}

\section{Main Results}
The following is our main result.
\begin{theorem}\label{thm: decompnospec}
For $d\geq 2$, the operator system $\op{Decomp}_d$ of decomposable matrices  does not admit a finite-dimensional realization, and the operator system $\op{Dpsd}_d$ of doubly positive matrices is not finitely generated.
\end{theorem}

We will prove Theorem \ref{thm: decompnospec} by showing that $\op{Decomp}_{2,2}$ is not a classical spectrahedron, and this by exhibiting a two-dimensional subspace on which $\op{Decomp}_{2,2}$ fails to fulfill even two necessary conditions for having a linear matrix inequality definition. However, the hard part lies in determining how the intersection of $\op{Decomp}_{2,2}$ with our subspace looks like. As the  convex hull of two easy sets, the intersection with a subspace is not necessarily the convex hull of the two subsets intersected with the subspace. This can be seen in Figure \ref{fig:pluspsd}. We solve this problem by using that $\op{Decomp}_{2,2}$ coincides with $\op{Bpsd}_{2,2},$ and examining the latter.
We start with some  intermediate results.
\begin{lemma}\label{lem: biquad}
For $a,b\in\mathbb{R}$, consider the following biquadratic form in the complex variables $x_1,x_2$ with matrix coefficients:
$$
p(x_1,x_2)=\begin{pmatrix}
1&0\\0&\frac{1}{4}
\end{pmatrix}
| x_1|^2+
\begin{pmatrix}
0&\frac{a}{4}\\\frac{a}{4}&0
\end{pmatrix}(\overline{x}_1x_2+x_1\overline{x}_2)+
\begin{pmatrix}
b&0\\0&\frac{1}{4}
\end{pmatrix}
| x_2|^2.
$$
Then $p$ is (globally)  positive semidefinite if and only if
\begin{align*}
(a,b)\in S_1\cup S_2
\end{align*}
where
\begin{align}\label{lem: biquadsets}
\begin{split}
& S_1:=\{(a,b)\in\mathbb{R}^2 \,|\, (b+1-a^2)^2-4b\leq 0\},\\
& S_2:=\{(a,b)\in\mathbb{R}^2\,| \, (b+1-a^2)^2-4b\geq 0, b\geq 0,a^2-b-1\leq 0\}.
\end{split}
\end{align}
\end{lemma}

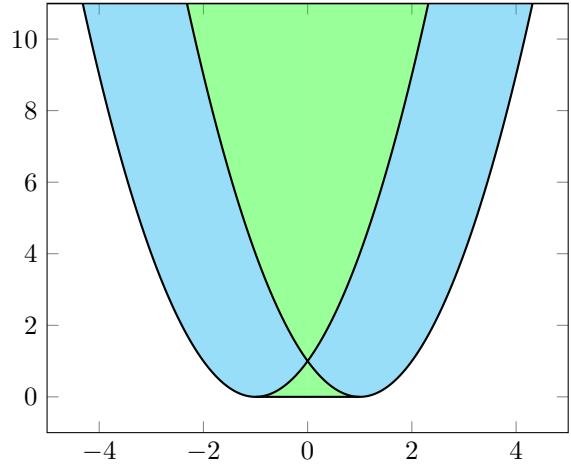
\begin{figure}
\begin{tikzpicture}
\begin{axis}[
      xmin=-5,xmax=5,
      ymin=-1,ymax=11,
      ]
	\plot[name path=A, thick,samples=200,domain=-5:5] {1-2*x+x^2};
	\plot[name path=B,thick,samples=200,domain=-5:5] {1+2*x+x^2};
	\plot[name path=X,draw=none,samples=200,domain=-5:5]{0};
		\plot[name path=Y,thick,samples=200,domain=-1:1]{0};
	\plot[name path=C,thick,samples=200,domain=-5:5]{11};
	\addplot[fill=cyan,opacity=.4] fill between [of=A and B, soft clip={domain=-5:5}];
		\addplot[fill=green,opacity=.4] fill between [of=B and X,soft clip={domain=-1:0}];
	\addplot[fill=green,opacity=.4] fill between [of=A and X,soft clip={domain=0:1}];
			\addplot[fill=green,opacity=.4] fill between [of=C and B,soft clip={domain=0:5}];
		\addplot[fill=green,opacity=.4] fill between [of=C and A,soft clip={domain=-5:0}];
	\end{axis}
\end{tikzpicture}
\caption{Section of the area $S_1\cup S_2$ of Lemma \ref{lem: biquad}. The blue part depicts $S_1$, the green one $S_2$.}\label{fig: biquad}
\end{figure}
\begin{proof}
First note that both global positivity of $p$ as well as  $(a,b)\in S_1\cup S_2$ requires $b\geq 0$, so we can assume this throughout the proof.

By multiplying/dividing  $p$ by $1/| x_2|^2$ one easily checks that
\begin{align*}
\forall x_1,x_2\in\C:p(x_1,x_2) & \geqslant 0\\
&\Longleftrightarrow\forall z\in\C:
p(z,1)=\begin{pmatrix}
1&0\\0&\frac{1}{4}
\end{pmatrix}
| z|^2+
\begin{pmatrix}
0&\frac{a}{2}\\\frac{a}{2}&0
\end{pmatrix}\op{Re}(z)+
\begin{pmatrix}
b&0\\0&\frac{1}{4}
\end{pmatrix}\geqslant 0.
\end{align*}

Next we convince ourselves that  $p(z,1)\geqslant 0$ for all  $z\in\R$ already implies $p(z,1)\geqslant 0$ for all $z\in\C$. Indeed for $\alpha,\beta\in\R$ and $z=\alpha+\mathrm{i}\beta$ we have
\begin{multline*}
p(z,1)=\begin{pmatrix}
1&0\\0&\frac{1}{4}
\end{pmatrix}
(\alpha^2+\beta^2)+
\begin{pmatrix}
0&\frac{a}{2}\\\frac{a}{2}&0
\end{pmatrix}\alpha+
\begin{pmatrix}
b&0\\0&\frac{1}{4}
\end{pmatrix}
\\=\underbrace{
\begin{pmatrix}
1&0\\0&\frac{1}{4}
\end{pmatrix}\alpha^2+
\begin{pmatrix}
0&\frac{a}{2}\\\frac{a}{2}&0
\end{pmatrix}\alpha+
\begin{pmatrix}
b&0\\0&\frac{1}{4}
\end{pmatrix}}_{=p(\alpha,1)\geqslant 0}+
\underbrace{\begin{pmatrix}
1&0\\0&\frac{1}{4}
\end{pmatrix}
\beta^2}_{\geqslant 0}\geqslant 0.
\end{multline*}
Thus it is enough to prove the statement for 
$$
\hat{p}(r)=\begin{pmatrix}
1&0\\0&\frac{1}{4}
\end{pmatrix}
r^2+
\begin{pmatrix}
0&\frac{a}{2}\\\frac{a}{2}&0
\end{pmatrix}r+
\begin{pmatrix}
b&0\\0&\frac{1}{4}
\end{pmatrix}
$$
and $r\in\R$.

Since $b\geq 0,$ we  have $\hat{p}(r)\geqslant 0$ if and only if $\op{Det}(\hat{p}(r))\geq 0$. A quick calculation shows that 
$$
\op{Det}\big(\hat{p}(r)\big)\geq 0\Leftrightarrow r^4+(b+1-a^2)r^2+b\geq 0.
$$
Substituting $r^2\mapsto s$ in the right hand side above, we obtain the polynomial $$q(s)=s^2+(b+1-a^2)s+b\in\R[s]$$ in the real variable $s$, of which we want to express positivity on $[0,\infty)$. So let $s_1,s_2$ denote the roots of $q$ and $\op{D}(q)$ the discriminant of $q$. Then it holds
\begin{align*}
\forall r\in\R:\ \op{Det}\big(\hat{p}(r)\big)\geq 0\Leftrightarrow\Big(\forall s\in\R:q(s)\geq 0 \lor \big(\op{D}(q)\geq 0\land\max\{s_1,s_2\}\leq 0\big)\Big).
\end{align*}
Clearly $q$ is nonnegative on the whole real line if and only if $\op{D}(q)\leq 0.$ Since  
\begin{align*}
\op{D}(q)=(b+1-a^2)^2-4b,
\end{align*}
 this is equivalent to $(a,b)\in S_1.$ For the second condition note that due to Vieta's formula for polynomials of degree two, it holds
$$
\max\{s_1,s_2\}\leq 0\Leftrightarrow (a^2-b-1=s_1+s_2\leq 0 \land b=s_1\cdot s_2\geq 0).
$$
Since these are the defining inequalities for  $S_2$, the proof is complete.
\end{proof}

In the following we use some 
basic definitions and results on semialgebraic and convex sets. The reader may consult the appendix (Section \ref{sec:app}) for  more detailed definitions and explanations.
\begin{proposition}\label{prop: nospec}
The convex cone $$\op{Decomp}_{2,2}=\op{Bpsd}_{2,2}\subseteq {\rm Her}_2(\C)\otimes {\rm Her}_2(\C)={\rm Her}_4(\C)$$ has a non-exposed face, and is not basic closed semialgebraic. In particular it is not a (classical) spectrahedron.
\end{proposition}
\begin{proof}
 For $a,b\in\R$ consider the matrix $M(a,b)$ defined as 
\begin{equation}\label{eq: matrixcounter}
\begin{pmatrix}
1 &0 & 0 & \frac{a}{4}\\
0& \frac{1}{4}&\frac{a}{4}&0\\
0& \frac{a}{4}&b&0\\
\frac{a}{4}&0 & 0&\frac{1}{4}
\end{pmatrix}=E_{11}\otimes\left(\begin{array}{cc}1 & 0 \\0 & \frac14\end{array}\right) +(E_{12}+E_{21})\otimes \left(\begin{array}{cc}0 & \frac{a}{4} \\\frac{a}{4} & 0\end{array}\right)+ E_{22}\otimes \left(\begin{array}{cc}b & 0 \\0 & \frac14\end{array}\right).
\end{equation}
Since $\op{Decomp}_{2,2}=\op{Bpsd}_{2,2}$ \cite{stormer,wor}, $M(a,b)$ is decomposable if and only if its corresponding biquadratic form with matrix coefficients
$$
p_{M(a,b)}(x_1,x_2)=\begin{pmatrix}
1&0\\0&\frac{1}{4}
\end{pmatrix}
| x_1|^2+
\begin{pmatrix}
0&\frac{a}{4}\\\frac{a}{4}&0
\end{pmatrix}(\overline{x}_1x_2+x_1\overline{x}_2)+
\begin{pmatrix}
b&0\\0&\frac{1}{4}
\end{pmatrix}
| x_2|^2
$$
in the complex variables $x_1,x_2$ is globally positive semidefinite. From Lemma \ref{lem: biquad} we get 
$$
\{(a,b)\in\R^2\,|\, M(a,b) \text{ is decomposable}\}= S_1\cup S_2
$$ 
where $S_1$ and $S_2$ are defined as in (\ref{lem: biquadsets}). One checks easily that $(-1,0)$ as well as $(1,0)$ are non-exposed faces of $S_1\cup S_2$, e.g.\ by computing the gradients of the boundary curves and showing that the only supporting hyperplane is defined by  $b=0$. 

Since $\{M(a,b)\,|\, a\in\R,b\in\R\}$ is obtained by intersecting $\op{Decomp}_{2,2}$ with  an affine subspace, Lemma \ref{lem: hyperplanes}  shows that also $\op{Decomp}_{2,2}$ has a non-exposed face. 

 Next we show that $\op{Decomp}_{2,2}$ is not basic closed. Since
$$
(b+1-a^2)^2-4b= (1+2a+a^2-b)(1-2a+a^2-b),
$$
the algebraic boundary $\partial_a(S_1\cup S_2)$  is given as the union of the zero sets of the polynomials $p_1(a,b)= 1+2a+a^2-b,p_2(a,b)=1-2a+a^2-b$ and $p_3(a,b)=b$, see also Figure \ref{fig: algbound}. Furthermore $S_1\cup S_2$ is a nonempty convex semialgebraic set with nonempty interior, and thus itself as well as its complement $\R^2\setminus(S_1\cup S_2)$ are regular. Therefore applying Proposition \ref{prop: sinn}  to the irreducible polynomial $p_1$ or $p_2$, and then  Lemma \ref{lem:basic}, shows that $\op{Decomp}_{2,2}$ is not basic closed semialgebraic.

Finally, both the non-exposed face as well as not being basic closed semialgebraic prevent $\op{Decomp}_{2,2}(=\op{Bpsd}_{2,2})$ from being a spectrahedron (\cite{Ramana1995},\cite{nepl}).
\end{proof}
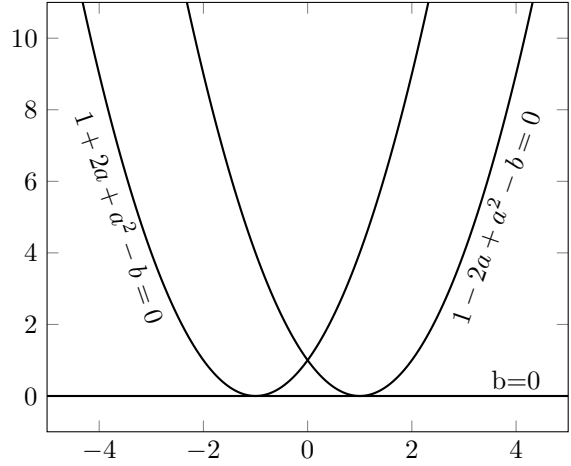
\begin{figure}
\begin{tikzpicture}
\begin{axis}[
      xmin=-5,xmax=5,
      ymin=-1,ymax=11,
      ]
	\plot[name path=A, thick,samples=200,domain=-5:5] {1-2*x+x^2};
	\plot[name path=B,thick,samples=200,domain=-5:5] {1+2*x+x^2};
	\plot[name path=X,thick,samples=200,domain=-5:5]{0};
	\node[] at (axis cs: 4,.45) {b=0};
	\node[rotate=-71] at (axis cs: -3.6,5) {$1+2a+a^2-b=0$};
		\node[rotate=71] at (axis cs: 3.6,5) {$1-2a+a^2-b=0$};
	\end{axis}
\end{tikzpicture}
\caption{The algebraic boundary of $S_1\cup S_2$ from  Lemma \ref{lem: biquad} and the proof of Proposition \ref{prop: nospec}}\label{fig: algbound}
\end{figure}
\begin{lemma}\label{lem: embed}
For $A,C\in{\rm Her}_2(\C)$ and $B\in{\rm Mat}_2(\C)$ let
$$
M=\begin{pmatrix}
A & B\\ B^* & C
\end{pmatrix}\in{\rm Her}_2(\C)\otimes{\rm Her}_2(\C)
$$
and define its lift $\hat{M}$ as
$$
\hat{M}:= E_{11}\otimes 
\begin{pmatrix}
A & \mathbf{0}\\
\mathbf{0} & \mathbf{0}
\end{pmatrix}+
E_{12}\otimes
\begin{pmatrix}
B & \mathbf{0}\\
\mathbf{0} & \mathbf{0}
\end{pmatrix}
+
E_{21}\otimes\begin{pmatrix}
B^* & \mathbf{0}\\
\mathbf{0} & \mathbf{0}
\end{pmatrix}
+E_{22}\otimes\begin{pmatrix}
C & \mathbf{0}\\
\mathbf{0} & \mathbf{0}
\end{pmatrix}\in{\rm Mat}_d(\C)\otimes{\rm Mat}_s(\C)
$$
where $E_{ij}\in{\rm Mat}_d(\C)$ denotes the usual matrix unit,  and $\mathbf{0}$ denotes zero matrices of suitable sizes. Then we have
\begin{align*}
M\in \op{Decomp}_{2,2}  & \Leftrightarrow \hat{M}\in \op{Decomp}_{d,s}  \Leftrightarrow \hat{M}\in \op{Bpsd}_{d,s}  \Leftrightarrow M\in \op{Bpsd}_{2,2}.
\end{align*}
\end{lemma}
\begin{proof}
Assume $M\in \op{Decomp}_{2,2}$ and write $M=X+Y$ for $X\in\op{Psd}_{2,2}, Y\in \op{Psd}_{2,2}^\Gamma$. By taking the lifts $\hat{X},\hat{Y}$ in the same way as defined for $M$, we immediately obtain $\hat{M}=\hat{X}+\hat{Y},$ showing $\hat{M}\in \op{Decomp}_{d,s}.$

If $\hat M\in \op{Decomp}_{d,s}$ then clearly  $\hat M\in \op{Bpsd}_{d,s}$. Next, if  $\hat M\in \op{Bpsd}_{d,s}$, then $M\in \op{Bpsd}_{2,2}$. This is easily seen by lifting   elementary tensors $v\otimes w\in\C^2\otimes\C^2$ to elementary tensors $\hat v\otimes \hat w\in\C^d\otimes \C^s$ (by filling $v$ and $w$ up with zeros), and observing that $$(v\otimes w)^*M (v\otimes w)=(\hat v\otimes \hat w)^*\hat M (\hat v\otimes \hat w)$$ holds. Finally, block positivity implies decomposability in ${\rm Her}_2(\C)\otimes {\rm Her}_2(\C)$ \cite{stormer,wor}. 
\end{proof}

We can now finally give the proof of our main result.
\begin{proof}[Proof of Theorem \ref{thm: decompnospec}]
 Fix $d\geq 2$. By applying Lemma \ref{lem: embed} we embed $\op{Decomp}_{2,2}$ into $\op{Decomp}_{d,2}$.  Since $\op{Decomp}_{2,2}$ is not a spectrahedron by Proposition \ref{prop: nospec}, it is immediate that this is also true for $\op{Decomp}_{d,2}$. In particular, $\op{Decomp}_d$ is not a free spectrahedron, i.e.\ not finite-dimensional realizable.
  From Theorem \ref{thm: helton} and Proposition \ref{prop:dualsys} ($iii$) is follows that $\op{Dpsd}_d$ is not finitely generated.
\end{proof}

\begin{figure}
\begin{tikzpicture}
\begin{axis}[
      xmin=-5,xmax=5,
      ymin=-1,ymax=11,
      ]
	\plot[name path=A, thick,samples=200,domain=1:5] {1-2*x+x^2};
	\plot[name path=B,thick,samples=200,domain=-5:-1] {1+2*x+x^2};
		\plot[name path=X,draw=none,samples=200,domain=-5:5]{0};
		\plot[name path=Y,thick,samples=200,domain=-1:1]{0};

	\plot[name path=C,thick,samples=200,domain=-5:5]{11};
	\plot[name path=Psd,thick,samples=200,domain=-2:2,color=red]{x^2/4};
	
	\addplot[thick,red] coordinates {(-2,1)(-2,10.97)};
	\addplot[thick,red] coordinates {(2,1)(2,10.97)};	
	
	\addplot[fill=cyan,opacity=.4] fill between [of=C and X,soft clip={domain=-1:1}];
	\addplot[fill=cyan,opacity=.4] fill between [of=B and C,soft clip={domain=-5:-1}];
	\addplot[fill=cyan,opacity=.4] fill between [of=A and C,soft clip={domain=1:5}];
	\addplot[fill=red,opacity=.55] fill between [of=Psd and C,soft clip={domain=-2:2}];
	
\end{axis}
\end{tikzpicture}
\caption{The sets $\op{Sep}_{2,2}=\op{DPsd}_{2,2}=\op{Psd}_{2,2}=\op{Psd}^{\Gamma}_{2,2}$ (red) and $\op{Decomp}_{2,2}=\op{Bpsd}_{2,2}$ (blue) in our two-dimensional subspace.}\label{fig:pluspsd}
\end{figure}
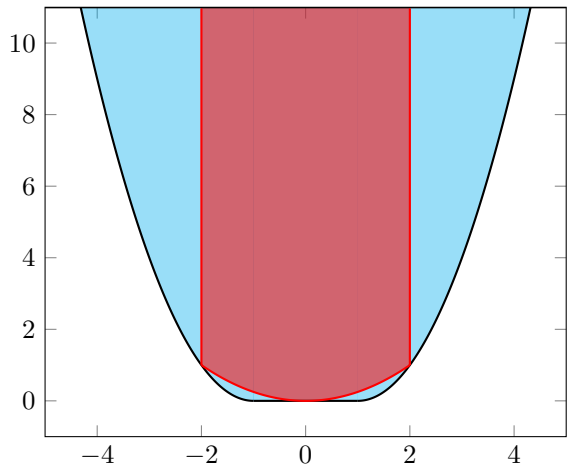

\begin{remark}
($i$) The above results indeed show that for all $d,s\geq 2$, the convex cones $\op{Decomp}_{d,s}$ and $\op{Bpsd}_{d,s}$ are not classical spectrahedra, since they have non-exposed faces and are not basic closed semialgebraic.  In particular, this reproves that $\op{Bpsd}_d$ is not finite-dimensional realizable and $\op{Sep}_d$ is not finitely generated. 

For $d+s>5,$ the cone $\op{Bpsd}_{d,s}$  is even known not to be the linear image of a spectrahedron (\cite{fawzi} combined with duality). However note that  $\op{Decomp}_{d,s}$ is the linear image of a spectrahedron, since it is the Minkowsi sum of two spectrahedra \cite{nepl}.

($ii$) Our result on $\op{Dpsd}_d$ can also be understood as follows. Whereare completely positive maps admit a finitary description by Choi's Theorem (as compressions of the identity), and the same is true for  completely copositive maps (as compressions of the transposition), no such finitary description is possible for doubly completely positive maps.

($iii$) In the abstract operator systems setup, our result shows that intersections of two finitely generated operator systems need not be finitely generated, and Minkowski sums of two operator systems with a finite-dimensional realization need not have a finite-dimensional realization.
\end{remark}

\section{Appendix}\label{sec:app}
Here we collect some definitions and results on semialgebraic and convex sets that we have used in the paper.
\begin{lemma}\label{lem: sumclosed}
Let $C,C_1,C_2\subseteq\R^d$ be closed, convex and salient cones, such that $C_1,C_2\subseteq C$. Then the Minkowski sum $C_1+C_2$ is closed.
\end{lemma}
\begin{proof}
Since $C$ is closed and salient we can find a compact convex  base $B$ of $C$ such that $0\notin B$. Clearly $B_1:=C_1\cap B$ and $B_2:=C_2\cap B$ are then compact bases of $C_1$ and $C_2$ respectively. Then $\tilde{B}:={\rm conv}(B_1\cup B_2)\subseteq B$  is a compact base of $C_1+C_2$ with $0\notin \tilde{B}$. Since every cone with compact such base is closed, see \cite{barvinok}, this proves the claim.
\end{proof}

 Recall that a {\it face} of a convex set $C\subseteq\mathbb{R}^d$ is a nonempty convex subset $F\subseteq C$, such that for $x,y\in C$ and $\lambda\in(0,1),\lambda x+(1-\lambda)y\in F$ implies $x,y\in F.$ A face $F$ of $S$ is \textit{exposed}, if $F=S$ or if there is an affine linear function $\ell$ on $\mathbb{R}^d$, with $\ell\geq 0$ on $C$ and
$$
F=\{a\in C\,|\, \ell(a)=0\}.
$$
In other words $F\neq C$ is exposed, if there exists a supporting hyperplane of $C$ that touches $C$ precisely in $F$. 

\begin{lemma}\label{lem: hyperplanes}
Let  $C\subseteq\R^d$ be a convex set, and $U\subseteq\R^d$ an affine subspace.
 If $C\cap U$ has a non-exposed face (in the space $U$), then so does $C$ in $\R^d$. \label{lem: convhyperplane}
\end{lemma}
\begin{proof} Let $F$ be a non-exposed face of $C\cap U$ in $U$.  Then there exists the unique smallest face  $\hat{F}\subsetneq C$ of $C$ containing $F$. Now assume $\hat{F}$ is exposed from $C$ by the supporting hyperplane $H$ in $\R^d$. Then $$F=\hat{F}\cap U=(C\cap H)\cap U=(C\cap U)\cap (U\cap H).$$ From $F\neq\emptyset$ we obtain $U\cap H\neq\emptyset,$ from $F\subsetneq C\cap U$ we obtain $U\not\subseteq H$.
Thus  $U\cap H$ is a hyperplane in $U$ that exposes $F$ from $C\cap U$, a contradiction.
 \end{proof}

\begin{proposition}[\cite{Ramana1995}]
Every face of a spectrahedron is exposed.
\end{proposition}

A {\it basic closed semialgebraic set} in  $\R^d$ is a set of the following form: 
$$\{ a\in \R^d\mid p_1(a)\geq 0,\ldots, p_r(a)\geq 0\}$$ where $p_1,\ldots, p_r\in\R[x_1,\ldots, x_d]$ are polynomials. A general {\it semialgebraic set} is a finite Boolean combination of basic closed sets.  

\begin{lemma}\label{lem:basic}
Let  $S\subseteq\R^d$ be a basic closed semialgebraic set, and $U\subseteq\R^d$ an affine subspace.
Then $S\cap U$ is a basic closed semialgebraic set in $U$.
\end{lemma}
\begin{proof}
This is obvious by restricting the defining polynomial of $S$ to $U$.
\end{proof}

{\it Quantifier elimination} in real closed fields \cite{prestel} implies that all level sets of the operator systems that we have considered in this paper are semialgebraic. However, as we have shown, they are not necessarily basic closed. 
To formulate a necessary condition we need the following definitions.
\begin{definition} Let $S\subseteq\mathbb{R}^d$ be a semialgebraic set. 
\begin{enumerate}[leftmargin=*]
\item[($i$)] The \textit{algebraic boundary} $\partial_a S$ of $S$ is the Zariski closure in $\mathbb{A}^n$ of its boundary $\partial S$ in the Euclidean topology.
\item[($ii$)] $S$ is called \textit{regular}, if it is contained in the closure of its interior (w.r.t.\ the Euclidean topology).
\end{enumerate}
\end{definition}

Note that every convex set with nonempty interior is regular, and so is its complement.

\begin{proposition}[\cite{sinn}]\label{prop: sinn}
Let $S\subseteq\mathbb{R}^d$ be a nonempty regular semialgebraic set, and suppose that its complement $\mathbb{R}^d\setminus S$ is also regular and nonempty. If the interior of $S$ intersects the algebraic boundary of $S$ in a regular point, then $S$ is not basic closed semialgebraic.
\end{proposition}

\section{Acknowledgements}
The first author gratefully acknowledges funding by the Austrian Academy of Sciences
(\"OAW), through a DOC-Scholarship.
\bibliography{ref.bib}
\end{document}